\documentclass[11pt]{article}
\usepackage{amsmath}
\usepackage[affil-it]{authblk}

\usepackage{lineno,hyperref}
\modulolinenumbers[5]
\usepackage{bm}
\usepackage{xcolor} 

\usepackage{mathtools,amsfonts,amssymb,amscd,xspace,bbm,mathabx,enumitem,empheq,xparse,amsthm}

\theoremstyle{plain}
\newtheorem{theorem}{Theorem}[section]

\newtheorem{lemma}[theorem]{Lemma}
\newtheorem{proposition}[theorem]{Proposition}
\theoremstyle{definition}
\newtheorem{definition}[theorem]{Definition}
\theoremstyle{remark}
\newtheorem{remark}[theorem]{Remark}
\theoremstyle{claim}
\newtheorem{claim}[theorem]{Claim}

\usepackage[margin=2.5cm]{geometry}
\usepackage{setspace}
\DeclarePairedDelimiter{\norm}{\lVert}{\rVert}
\newcommand{\half}{\frac{1}{2}}
\newcommand{\Rd}{\mathbb{R}^d}
\newcommand{\RdxRd}{\Rd \times \Rd}
\newcommand{\RdxRdxRd}{\Rd \times \Rd \times \Rd}

\newcommand{\oS}{\overline{S}}
\newcommand{\oPFP}{\overline{PFP}}
\newcommand{\ooPFP}{\overline{\overline{PFP}}}

\newcommand{\ip}[2]{\langle #1\ , #2 \,\rangle}

\newcommand{\supp}{\textrm{supp}}

\newcommand{\R}{\mathbb{R}}
\newcommand{\N}{\mathbb{N}}

\newcommand{\extendedsub}[2]{{\partial}#1(#2)}
\newcommand{\strongsub}[2]{\boldsymbol{\partial}#1(#2)}
\newcommand{\minstrongsub}[2]{\boldsymbol{\partial}^0#1(#2)}

\newcommand{\tlambda}{\tilde{\lambda}}

\DeclareMathOperator*{\argmin}{arg\,min}
\allowdisplaybreaks

\begin{document}

\title{Gradient Flows for Probabilistic Frame Potentials in the Wasserstein Space }
\author{Clare Wickman}
\affil{Johns Hopkins University Applied Physics Laboratory}
%
\author{Kasso A. Okoudjou}
\affil{Department of Mathematics, Tufts University}
\maketitle

{\bf Keywords:} frame potential, probabilistic frames, optimal transport, Wasserstein space, gradient flows

\begin{abstract}
In this paper we bring together some of the key ideas and methods of two disparate fields of mathematical research, frame theory and optimal transport, using the methods of the second to answer questions posed in the first. In particular, we construct gradient flows in the Wasserstein space $P_2(\Rd)$ for a new potential, the tightness potential, which is a modification of the probabilistic frame potential. It is shown that the potential is suited for the application of a gradient descent scheme from optimal transport that can be used as the basis of an algorithm to evolve an existing frame toward a tight probabilistic frame.
\end{abstract}


\section{Introduction}\label{sec:intro}

We consider the family of probabilistic $p$-frame potentials defined on a subset of the Borel probability measures $P(\Rd)$ by $$PFP_p(\mu) = \iint_{\RdxRd} \left|\ip{x}{y}\right|^pd\mu(x)d\mu(y), \quad\forall\mu\in P(\Rd),\quad p>0$$ and define a new, related ``tightness" potential 
$$TP(\mu) = \frac{PFP_2(\mu)}{M_2(\mu)}-\frac{M_2(\mu)}{d},$$ defined to be zero when $\mu$ is a delta mass at the origin.

The study of global minimizers of energies of this form  is of particular interest to  many areas of mathematics including frame theory, discrepancy theory, and  design theory \cite{BilykPark, BilykDaiMatzke, KassoMartinOverview, MartinKassoPframe,  Venkov}. In particular, it was proved in \cite{MartinKassoPframe} that,  when restricted to probability measures on the unit sphere, the minimizers of $PFP_p$ are all discrete measures when $0<p<2$, and it is conjectured that this is in fact the case for all $0<p\neq 2k$ \cite{BilykPark}.  One of the outcomes of this paper is an algorithm that will approximate the minimizers of $PFP_p$  when $p=2k$ is an even integer.
More generally,   in the finite discrete  setting with uniform distributions over sets of $N$ atoms, $N\geq d$, minimizers of these functionals form classes of special  configurations including finite unit norm tight frames ($p=2$), equiangular tight frames, spherical $p$-designs ($p$ even), and Grassmannian frames ($p=\infty$).  Especially when Parseval (with frame constant $A=B=1$ below), such constructions are relevant in coding theory \cite{GoyalKovacevic,LiuDaiLuo}, quantum measurement \cite{appl2005,ZippMatlin,RenBluScoCav2004,Oviedo}, and statistical shape analysis \cite{EhlerGalanis}, including sparse principal component analysis \cite{LuZhangPCA}, the orthogonal procrustes problem \cite{EldenParkProcrustes,OviedoProcrustes}, pattern recognition \cite{KokiopoulouChenSaad}, 1-bit compressive sensing \cite{BoufounosBaraniuk1bitCS}, and color image restoration and conformal mapping construction \cite{OsherLai}. For many of these problems the link is optimization over the Stiefel manifold of Parseval frames.

On the other hand, the potential $PFP_p$ can be viewed outside of the frame context as a special case of functionals of the form $\iint_{\R^d\times \R^d}f(\ip{x}{y}) d\mu(x)d\mu(y)$ where $f$ is a real-valued function describing the nature of the potential, e.g., attractive vs repulsive. Unlike the interaction potential, a class of functionals of the form $\iint_{\R^d\times \R^d}W(x-y) d\mu(x)d\mu(y)$ where the interaction depends on the distance between relative positions of points in the support, the dynamics dictated by the probabilistic frame potentials depend upon the degree of coherence among points. With support restricted to the sphere, the $2k$-probabilistic frame potentials can be written as interaction potentials (albeit with $W(0)=1$) and should be $\lambda$-locally convex and support the construction of gradient flows on the sphere, as in \cite{carrilloslepcevwu}. A number of interesting papers have been written on the properties of interaction potentials, including the geometry of their minimizers \cite{CarrilloFigalli,balague2013dimensionality}.


On $\Rd,$ these results do not apply; the frame potential no longer takes the usual form of an interaction potential. Consequently, in Section~\ref{sec: the tightness potential} we introduce a new functional, the \emph{tightness potential}, whose minimizers are characterized as the tight frames (Proposition~\ref{tp_pot_bound_tp_op}). In particular, our main result, Theorem \ref{existence_flows_tp}, uses a formalism described in~\cite{AGS2005} to establish the well-posedness of the minimization of the tightness potential via gradient flows on probability spaces through the application of a variational scheme which gives approximations to the flow.  In  the process we prove some regularity properties of the tightness potential, which can be generalized for higher values of $p$. After some preliminaries on frame theory and optimal transport, we briefly explain this approach.

A second consequence of enlarging the set of measures of interest from $P(S^{d-1})$ to $P(R^d)$ is that certain compactness results no longer apply. Therefore, we restrict the domain of our functional 
to measures with uniformly bounded moments of some order greater than 2, see Section \ref{sec: the tightness potential}.  This restriction has limited effects numerically since there is already a natural limitation on frames that are machine representable, and a similar restriction has been considered in other works such as \cite{Craig2017}.

Finally, we emphasize that the approach we lay out to the tightness potential minimization, i.e., via approximations of the flow, is implementable algorithmically. While the results of that algorithm are, as expected, a path to a numerical approximation to the closest tight frame, modifications, such as the addition of a penalty term, can be applied to the potential which lead to other close tight frames. A similar algorithmic approach, not based on gradient flows on probability spaces, was described in \cite{AutoTuning}, though restricted to unit-norm frames.

 
\subsection{Frames and probabilistic frames}\label{sec:subsec1.1}

Frames are redundant spanning sets of vectors or functions that can be used to represent signals in a faithful but nonunique way and that provide an intuitive framework for describing and solving problems in coding theory and sparse representation. We refer to \cite{Christensen2003, CasKut2013, OkoudjouFiniteFrame} for more details on finite frames and their applications.

A set 
 $\Phi=\{\varphi_i\}_{i=1}^N\subset\Rd$ is a frame if and only if there exist frame bounds $0<A\leq B <\infty$ such that $$ \forall x\in\Rd,\quad A\norm{x}^2\leq\sum\limits_{i=1}^N\ip{x}{\varphi_i}^2\leq B\norm{x}^2.$$ When $A=B$ we say that $\Phi$ is tight for $\Rd$. If, in addition, $\|\varphi_i\|=1$ for all $i=1, \hdots, N$, we say that the frame is a finite unit-norm tight frame (FUNTF), and each $x\in \R^d$ can be written as $$x=\tfrac{d}{N}\sum_{i=1}^N \ip{x}{\varphi_i}\varphi_i.$$ FUNTFs were characterized by Benedetto and Fickus as the minimizers of the discrete analogue of $PFP_2$ \cite{BenedettoFickus}, and constitute a class of frames sought in certain applications such as coding theory. 


In \cite{MartinKassoPframe,OktayThesis}, it is shown that the minimizers of another functional called the $p$-frame potential are precisely the equiangular FUNTFs, those for which the mutual coherence between distinct frame elements is constant. While algebraic approaches exist for the construction of FUNTFs, such as those outlined in \cite{NateMattJameson, AutoTuning, SpectralTetris}, the existence of the potentials mentioned above suggests that variational methods for construction of tight frames and FUNTFs might complement these methods. 
Indeed, the idea of using differential calculus to find useful frames, which we employ, is not new. Differential calculus approaches \cite{AutoTuning, BodmannCasazza2010} and flows over orthonormal bases \cite{BenedettoKebo} have been applied to various related problems, such as quantum detection.  However, the setting of probabilistic frames in the Wasserstein space will allow the construction of much more general gradient flows for frame potentials because of the sophisticated machinery which has been developed for this space.

A probabilistic frame $\mu$ for $\Rd$ is a probability measure on $\Rd$ for which there exist constants $0<A\leq B<\infty$ such that for all $x\in\Rd$, $$A\norm{x}^2\leq\int_{\Rd}\ip{x}{y}^2d\mu(y)\leq B\norm{x}^2.$$ When we can choose $A=B$ we say that $\mu$ is a tight probabilistic frame. Developed in a series of papers \cite{KassoMartinOverview, MartinKassoPframe, MartinRTF, OkoudjouWickman}, probabilistic frames generalize finite frames.  It follows that a  frame   $\Phi=\{\varphi_i\}_{i=1}^N$ can be identified with the probabilistic frame $\mu_{\Phi}(x)=\tfrac{1}{N}\sum_{i=1}^N  \delta_{\varphi_i}(x).$

\subsection{Probabilistic frames and the Wasserstein space}\label{sec:subsec1.2}
 
A natural realm in which to explore probabilistic frames and flows thereof is that of optimal transport theory. For probabilistic frames, we consider the space of probability measures with finite second moments, or, more generally, with finite $p$-th moments, $P_p(\Rd)$:
$$M_p(\mu):=\int_{\Rd}\norm{x}^pd\mu (x) < \infty.$$ 
Define the support of a probability measure $\mu$ on $\Rd$ as the set:
$$\supp(\mu):=\left\{x\in \Rd \textrm{ s.t. for all open sets } U_x \textrm{ containing } x, \mu(U_x)>0\right\}.$$ 
By \cite[Theorem 5]{KassoMartinOverview}, $\mu$ is a probabilistic frame if and only if it has finite second moment and the linear span of its support is $\Rd$.  This characterization can be restated in terms of the probabilistic frame operator for $\mu$, $S_{\mu},$ which for all $y\in\Rd$ satisfies:
$$S_{\mu}y=\int_{\Rd}\ip{x}{y}\,x\,d\mu(x).$$
Equating $S_{\mu}$ with its matrix representation $\int_{\Rd}xx^{\top}d\mu(x)$, the requirement that the support of $\mu$ should span $\Rd$ is equivalent to this matrix's being positive definite. 

This way of viewing probabilistic frames leads us naturally to the Wasserstein space of probability measures with finite second moment, $P_2(\Rd),$ a metric space with distance defined by an optimal transport problem.  The ($p$-)Wasserstein distance, $W_p$ between two probability measures $\mu$ and $\nu$ on $\Rd$, defined for $p\geq 1$, is: $$W_p^p(\mu,\nu):=\inf\left\{\iint_{\RdxRd}\norm{x-y}^pd\gamma(x,y): \gamma \in \Gamma(\mu,\nu)\right\},$$ 
where $\Gamma(\mu,\nu)$ is the set of all joint probability measures $\gamma$ on $\RdxRd$ such that for all $B_1,B_2 \in \mathcal{B}(\Rd)$, $\gamma(B_1\times\Rd)=\mu(B_1)$ and $\gamma(\Rd\times B_2)=\nu(B_2)$. One may also write this in terms of projection mappings as $\pi_{\hash}^1 \gamma =\mu$ and $\pi_{\hash}^2 \gamma=\nu.$ The notation $\Gamma_0(\mu,\nu)$ is used to indicate the set of all joint probability measures achieving the infimum above. 

The Monge-Kantorovich optimal transport problem is the search for the set of joint measures which induce the infimum; any such joint distribution is called an optimal transport plan.  In the quadratic case, when $\mu$ does not give mass to sets of dimension at most $d-1$, then $$W_2^2(\mu,\nu):=\inf_{T}\left\{\int_{\Rd}\norm{x-T(x)}^2d\mu(x): T_{\#}\mu=\nu\right\},$$ where $T_{\#}\mu=\nu$ indicates that $T$ is a deterministic transport map (or deterministic coupling): i.e., for all $\nu$-integrable functions $\phi$, $$\int_{\Rd}\phi(y)d\nu(y)=\int_{\Rd}\phi(T(x))d\mu(x).$$

Equipped with the $p$-Wasserstein distance, $(P_p(\Rd),W_p)$ is a complete, separable metric (``Polish'') space. In fact the set of measures with discrete, finite support is dense in $P_p(\Rd)$ (c.f.,\cite{OkoudjouCheng}). Convergence in $P_p(\Rd)$ is the usual narrow convergence of probability measures, combined with convergence of the $p$th moments.  Alternatively, enlarging the set of allowable test functions, a sequence of measures $\{\mu_n\} \subset P_p(\Rd)$ is said to converge weakly in $(P_p(\Rd),W_p)$ to $\mu\in P_p(\Rd)$ if for all continuous functions $\phi$ with $$|\phi(x)|\leq C(1+\norm{x-x_0}^p),$$ for some $C>0$ and some $x_0\in \Rd$, $$\int_{\Rd}\phi(x)d\mu_n(x)\rightarrow\int_{\Rd}\phi(x)d\mu(x).$$

Much effort has gone into developing a rich theory of gradient flows on the $p$-Wasserstein spaces, with weak solutions to flows based on the theory of $p$-absolutely continuous curves (e.g., \cite{AGS2005, GKP}). Many PDEs can be reformulated as energy minimization problems in these spaces (e.g., \cite{CDFLS, JKO}). For our purposes, the technical basis for weak solutions provided by \cite{AGS2005} will be enough, although we will refer to intuitions and certain reformulations provided by \cite{GKP}. 

\subsection{Gradient Flows and the Minimizing Movement Scheme}\label{sec:subsec1.3}

Of the four approaches given in \cite{AGS2005} for establishing the well-posedness of a minimization problem in the Wasserstein spaces, we lead the reader through the one best suited to our potential, first defining some of the necessary gradient flow machinery, the subject of \cite{AGS2005,Villani2009}. For specialists in optimal transport, many of the definitions related to subdifferentials in the Wasserstein space will likely prove unnecessary, but we include it to make the paper self-contained and to make it accessible for those coming from the frame community. 

If $F:P_2(\Rd)\rightarrow(-\infty,\infty]$ is a functional on the 2-Wasserstein space then the domain of $F$ is \\$D(F)=\{\mu\in P_2(\Rd): F(\mu)<\infty\}$, and $F$ is proper if $D(F)$ is nonempty. To define gradient descent for functionals on such a space, we must have a notion of absolute continuity and metric derivative for curves in the $2$-Wasserstein space.

\begin{definition}\emph{\cite[Definition 1.1.1, Theorem 1.1.2, Definition 1.2.4]{AGS2005}}\label{metric_slope}
\\
A curve $\sigma_t:(a,b)\rightarrow P_2(\Rd)$ is $2$-absolutely continuous if there exists $\beta\in L^2((a,b))$ such that $$W_2(\sigma_t,\sigma_s)\leq\int_s^t\beta(\tau)d\tau\quad\textrm{for all }a<s<t<b.$$  The space of such curves is denoted $AC^2((a,b);P_2(\Rd)),$ and for such $\sigma\in AC^2(a,b;P_2(\Rd))$, the metric derivative $|\sigma'|(t):=\lim_{s\rightarrow t}\frac{W_2(\sigma_t,\sigma_s)}{|t-s|}$ exists for $\mathcal{L}^1$-a.e. $t\in (a,b)$. The metric slope $|\partial\phi|(\mu)$ of a functional $\phi:P_2(\Rd)\rightarrow (-\infty,\infty]$ at $\mu$ is given by 
\begin{equation}|\partial\phi|(\mu)=\limsup_{W_2(\mu,\nu)\rightarrow 0}\frac{(\phi(\mu)-\phi(\nu))^+}{W_2(\mu,\nu)},\label{metric_slope_eqn}
\end{equation} where $u^+=\max(0,u).$
\end{definition}

\begin{remark}
In what follows, the inclusion symbol $i$ is used for deterministic couplings in the sense that for a function $G:\RdxRd\rightarrow\R,$ the joint plan given by $\gamma = (i, f)_{\hash}\mu$ for some $f:\Rd\rightarrow\Rd$ behaves thus: $$\iint_{\RdxRd}G(x,y)d\gamma(x,y)=\int_{\Rd}G(x,f(x))d\mu(x).$$
\end{remark}

\begin{definition}\label{strong_p_subdifferential}\emph{\cite[The extended (strong) subdifferential, Definition 10.3.1]{AGS2005}}
Let $\phi:P_2(\Rd)\rightarrow(-\infty,\infty]$ be a proper and lower semi-continuous functional, and let $\mu^1\in D(\phi)$. Then $\gamma\in P_{2}(\RdxRd)$ belongs to the extended Fr\'{e}chet subdifferential $\extendedsub{\phi}{\mu^1}$ if $\pi^1_{\hash}\gamma=\mu^1$ and for any $\mu^3\in P_2(\Rd)$
$$\phi(\mu^3)-\phi(\mu^1)\geq\inf_{\nu\in\Gamma_0(\gamma,\mu^3)}\iiint\ip{x_2}{x_3-x_1}d\nu + o(W_2(\mu^1,\mu^3)).$$

\noindent We further say that $\gamma$ belongs to the extended \textbf{strong} Fr\'{e}chet subdifferential $\strongsub{\phi}{\mu^1}$ if for every $\nu\in\Gamma(\gamma,\mu^3)$, it satisfies the stronger condition
\begin{equation}\label{strong_frechet}
\phi(\mu^3)-\phi(\mu^1)\geq \iiint\ip{x_2}{x_3-x_1}d\nu + o(C_{2,\nu}(\mu^1,\mu^3)),
\end{equation}
where $C_{2,\nu}(\mu^1,\mu^3)$ is the pseudo-distance given by the cost 
\begin{equation}\label{pseudodistance}
C^2_{2,\nu}(\mu^1,\mu^3)=\iiint\norm{x_1-x_3}^2d\nu(x_1,x_2,x_3).
\end{equation}
\end{definition}

The subdifferential $\extendedsub{\phi}{\mu^1}$ is a multivalued operator which extends curves of maximal slope to a nonsmooth setting. The version defined here is termed extended because it applies to general (not merely regular) probability measures, where a reference measure (e.g., $\mu_1$ above) cannot be pushed to a testing measure (e.g., $\mu_3$ above) by a deterministic transport plan (i.e., $\gamma$ above may not be given by some $(i, t^{\mu^3}_{\mu_1})_{\hash}\mu_1$), where $t^{\mu^3}_{\mu_1}$ denotes a map pushing $\mu_1$ to $\mu_3$. For this reason, the intuition of a Euclidean distance vector (suggested by the term $x_3-x_1$) cannot be analogized here via a displacement map using $t^{\mu^3}_{\mu_1}-i$. Finally, we have the ability to demand a strong subdifferential, meaning that we are free to perturb a minimizer by an arbitrary transport map, not just the optimal one. We also note that, properly speaking, the term ``subdifferential'' defined below refers to a set; however, this term is also commonly used to refer to elements of that set. Following the literature, we will use the term interchangeably.

Next we define regularity, a property which provides closure of the subdifferential, gives existence of a unique minimal selection in the subdifferential (defined in Lemma \ref{minimal_selection}), and implies the lower semicontinuity of the metric slope.

Below, 
\begin{align}\label{partial_moment}
|\gamma|_{j,2}^2=\iint_{\RdxRd}\norm{x_j}^2d\gamma(x_1,x_2),\quad j=1,2.
\end{align}
\begin{definition}\emph{\cite[Regular functionals, Definition 10.3.9]{AGS2005}, \cite[Definition 6]{Bonnet}}\label{regular_functional}
A proper, lower semicontinuous functional $\phi:P_2(\Rd)\rightarrow(-\infty,\infty]$ is regular if whenever the measures $\mu_n$ and the elements of their strong subdifferentials $\gamma_n\in\strongsub{\phi}{\mu_n}$ satisfy
$$\phi(\mu_n)\rightarrow\varphi\in\R,\quad\mu_n\rightarrow\mu\quad\textrm{in }P_2(\Rd),$$ $$\sup_{n}|\gamma_n|_{2,2}<\infty,\quad\gamma_n\rightarrow\gamma\quad\textrm{in }P(\Rd\times\Rd)$$
then $\gamma\in\extendedsub{\phi}{\mu}$, and $\varphi=\phi(\mu)$.
\end{definition}


Since the subdifferential $\partial \phi(\mu)$ of a functional $\phi$ at $\mu$ in $P_2(\Rd)$ may be multivalued, we define a gradient flow in terms of a differential inclusion:                      

\begin{definition}\emph{\cite[Definition 11.1.1]{AGS2005}}
Given a map $\mu_t\in AC_{loc}^2((0,\infty); P_2(\Rd))$ with  $v_t\in Tan_{\mu_t}P_2(\Rd)$ the velocity vector field of $\mu_t$, $\mu_t$ is a solution of  the gradient flow equation
\begin{equation}\label{grad_flow}
v_t\in -\partial \phi(\mu_t),\quad t>0
\end{equation}
if $(i, -v_t)_{\hash}\mu_t\in\strongsub{\phi}{\mu}$ for a.e. $t>0.$ 
\end{definition}

Finally, we have the following result about the minimal selection in the strong subdifferential related to the metric slope of Definition \ref{metric_slope}. This element is important because it often satisfies both variational inequalities along optimal tranport plans and directional derivative inequalities along more general plans.

\begin{lemma}\emph{\cite[Minimal selection, Theorem 10.3.11]{AGS2005}}\label{minimal_selection}
Let $\phi:P_2(\Rd)\rightarrow(-\infty,\infty]$ be a regular, lower semi-continuous functional, bounded below, and let $\mu\in D(\bm{\partial}\phi),$ i.e., $\bm{\partial}\phi(\mu)\neq\emptyset.$ Then there exists a unique plan $\gamma_0\in\extendedsub{\phi}{\mu}$ which attains the minimum $$|\gamma_0|_{2,2}=min\{|\gamma|_{2,2}:\gamma\in\strongsub{\phi}{\mu}\}=|\partial\phi|(\mu).$$ $\gamma_0$ will be denoted by the symbol $\bm{\partial}^0\phi(\mu).$
\end{lemma}

One approach to solving the gradient flow equation in the Wasserstein space is to draw an analogy with the usual setting of gradient flows on a Riemannian manifold and perform a time discretization of the steepest descent equation. This scheme was pioneered by \cite{JKO}, and its convergence is equivalent to the above formulation of the gradient flow, as laid out in \cite[Chapter 11]{AGS2005}. To describe this scheme, we will follow \cite{Kamalinejad} and \cite[Chapter 11.1.3]{AGS2005}.
\begin{definition}\emph{The Minimizing Movement Scheme}\label{min_mov}
Assume the following:
\begin{itemize}[label={}] 
\item Let $\phi:P_2(\Rd)\rightarrow (-\infty,\infty]$ be a proper, lower semicontinuous functional such that  for some $\tau^* > 0$ $$\nu\mapsto \Phi(\tau,\mu;\nu):=\frac{1}{2\tau}W_2^2(\mu,\nu)+\phi(\nu)$$ admits a minimum point for all $\tau\in(0,\tau^*)$ and $\mu\in P_2(\Rd).$  \label{min_moment_cond}
\end{itemize}
Fix a measure $\mu_0\in P_2(\Rd).$  Given any step size $\tau>0$, we can partition $(0,\infty]$ into $\bigcup_{n=1}^{\infty}((n-1)\tau,n\tau].$  For a given family of initial values $M_{\tau}^0$ such that $$M_{\tau}^0\rightarrow \mu_0\quad\textrm{in } P_2(\Rd),\quad\phi(M_{\tau}^0)\rightarrow\phi(\mu_0)\quad \textrm{as }\tau\downarrow 0$$ we can define for each $\tau\in (0,\tau^*)$ a family of sequences $\{M_{\tau}^n\}_{n=1}^{\infty}$ satisfying $$M_{\tau}^n=\argmin_{\nu\in D(\phi)}\Phi(\tau, M_{\tau}^{n-1};\nu),$$ where the choice of $M_{\tau}^n$ may not be unique, but such a measure will always exist. Then the piecewise constant interpolant path in $P_2(\Rd),$ $$\overline{M}_{\tau}(t):=M_{\tau}^n,\quad t\in ((n-1)\tau,n\tau],$$ is termed the discrete solution. 
A curve $\mu$ will  be a Generalized Minimizing Movement for $\Phi$ and $\mu_0$ if there exists a sequence $\tau_k\downarrow 0$ such that $$\overline{M}_{\tau_k}(t)\rightarrow\mu_t\quad\textrm{narrowly in }P(\Rd)\textrm{ for every }t>0,\quad\textrm{ as }k\rightarrow\infty.$$
For $\mu_0\in D(\phi)$, by a compactness argument, this solution always exists and is an absolutely continuous curve $\mu\in AC_{loc}^2([0,\infty);P_2(\Rd)).$
\end{definition}
As observed by \cite{AGS2005} to illustrate the goal of this method, if we can restrict the domain of the functional $\phi$ and its gradient to the regular measures, then we can define a sequence of optimal transport maps $T_{\tau}^n$ pushing $M_{\tau}^n$ to $M_{\tau}^{n-1}.$  Then the discrete velocity vector can be defined as $$V_{\tau}^n:=\frac{T_{\tau}^n-i}{\tau}\in\partial\phi(M_{\tau}^n),$$
which is an implicit Euler discretization of \eqref{grad_flow}.  The piecewise constant interpolant $$\overline{V}_{\tau}(t):=V_{\tau}^n\quad\textrm{for }t\in((n-1)\tau,n\tau],$$ converges distributionally in $\Rd\times(0,\infty)$ up to subsequences to a vector field which solves the continuity equation.  The problem which remains is proving that this vector field is also a solution of \eqref{grad_flow}.

For regular functionals, without having to restrict ourselves to the convex case or to regular measures, it can be shown that this convergence occurs; the following theorem gives sufficient conditions for this convergence. This is particularly important, given that we want to apply the result to discrete measures such as embeddings of finite frames. Indeed, several recent papers have noted that minimizers of related potentials are discrete \cite{CarrilloFigalli} or have empty interior \cite{BilykPark}.  We therefore rely on this for our main result.

\begin{lemma}\emph{\cite[Theorem 11.3.2.]{AGS2005}}\label{existence_flows_ags}
Let $\phi:P_2(\Rd)\rightarrow(-\infty,\infty]$ be a proper and lower semicontinuous  regular functional (with respect to the narrow convergence) with sublevel sets that are compact in $P_2(\Rd)$.  Then for every initial datum $\mu_0\in D(\phi)$, each sequence of discrete solutions $\overline{M}_{\tau_k}$ of the variational scheme admits a (not relabeled) subsequence such that
\begin{enumerate}
\item $\overline{M}_{\tau_k}(t)$ narrowly converges in $P(\Rd)$ to $\mu_t$ locally uniformly in $[0,\infty)$, with \\$\mu_t\in AC_{loc}^2([0,\infty);P_2(\Rd))$.
\item $\mu_t$ is a solution of the gradient flow equation $$v_t=-\partial^0\phi(\mu_t),\quad\norm{v_t}_{L^2(\mu_t;\Rd)}=|\mu^{'}|(t), \textrm{ for a.e. } t>0$$
with $\mu_t\rightarrow\mu_0$ as $t\downarrow 0,$ where $v_t$ is the tangent vector to the curve $\mu_t.$
\item The energy inequality $$\int_a^b\int_{\Rd}|v_t(x)|^2d\mu_t(x)dt+\phi(\mu_b)\leq \phi(\mu_a)$$ holds for every $b\in [0,\infty)$ an $a\in [0,b)\setminus\mathcal{N}$, where $\mathcal{N}$ is a $\mathcal{L}^1$-negligible subset of $(0,\infty).$
\end{enumerate}
\end{lemma}

To interpret the notation used in item 2, we note that any $\gamma_0\in\minstrongsub{\phi}{\mu}$ will be concentrated on the graph of the transport map $(i,-v_t)_{\hash}\mu$ \cite[Theorem 11.1.3]{AGS2005}. The compactness of the sublevel sets provides for the compactness of discrete trajectories and therefore the existence of limit trajectories in the minimizing movement scheme as the step size $\tau\downarrow 0.$


\subsection{Outline of the paper}\label{sec:subsec1.4}

This paper establishes the well-posedness of the problem of minimizing an energy that we term the tightness potential (related to the frame potential) via gradient flows over the space $P_2(\Rd)$. This quantity is introduced in sections \ref{sec: the frame potential} through \ref{sec: the tightness potential}. Section \ref{sec: construction} establishes the characteristics of this potential needed to show the well-posedness of this problem, according to the criteria laid out in \cite[Section 11.3.1]{AGS2005} and summarized in brief above. The main result can be found in section \ref{sec: well-posedness}, Theorem \ref{existence_flows_tp}, namely that gradient flows exist for a potential measuring the tightness of a frame and that for this potential an approximating subsequence can be constructed via a variational scheme, and which satisfies an energy inequality. Notably, this variational scheme is sufficient to build an algorithm to construct paths to numerical approximations to tight frames.
Finally, section \ref{sec: the fourth and higher potentials} discusses extensions of this approach to algorithm-building to higher-order frame potentials that are of interest to the frame community.


%

\section{Gradient descent for Probabilistic Frame Potentials}\label{sec: frame_forces}

\subsection{The frame potential}\label{sec: the frame potential}
It is natural to begin the discussion of frame forces with the frame potential for finite frames and the analogous quantity for probabilistic frames.
The probabilistic frame potential for a probabilistic frame $\mu$ is given by
\begin{equation}\label{prob_frame_pot}
PFP(\mu)=\iint_{\RdxRd}\ip{x}{y}^2d\mu(x)d\mu(y)
\end{equation}
%

And we define the scaled probabilistic frame potential as
\begin{align}\label{scaled_prob_frame_pot}
\overline{PFP}(\mu) = \left\{ \begin{array}{cc} 
                \frac{PFP(\mu)}{M_2(\mu)} & \hspace{5mm} \mu\neq\delta_{\{0\}} \\
                0 & \hspace{5mm} \text{else} \\
                \end{array} \right.
\end{align}
\noindent where $\delta_{\{0\}}$ denotes the delta mass at the origin. 

In what follows, we explore several functionals on the space $P_2(\Rd)$ related to questions in frame theory, starting with the probabilistic frame potential.  For equal-norm frames restricted to spheres, this potential would be sufficient to identify tightness. In order to generalize, we tweak the scaled version to a quantity we term the tightness potential. 


\subsection{Locating tight probabilistic frames}\label{sec: locating tight probabilistic frames}

Some further analysis is needed before using the frame potential to find probabilistic tight frames.  In the following propositions and lemmas, we narrow our search space, establish a lower bound on how close the nearest probabilistic tight frame can be. We also show that, as in the finite case, the frame potential is indeed a crucial quantity in constructing gradient flows that will lead us to tight probabilistic frames. In fact, for a given probabilistic frame $\mu,$ we have control on the spectrum of the frame operators of the measures nearby in $P_2(\Rd),$ as seen in the next result.

Throughout the paper, we make use of a pseudo-distance generalizing (and bounding) the Wasserstein distance between two measures $\mu$ and $\nu\in P_2(\Rd).$ Given some $\gamma\in P_2(\Rd\times\Rd\dots\times\Rd)$ s.t. $\pi^i_{\hash}\gamma=\mu$ and $\pi^j_{\hash}\gamma=\nu,$ we define $C^2_{2,\gamma}(\mu,\nu):=\int\dots\int\norm{x_i-x_j}^2d\gamma(x_1,x_2,\dots).$
\begin{proposition}\label{frame_op_dist}
Suppose $\{\nu_n\}$ is a sequence converging to some $\mu\in P_2(\Rd).$ 
\begin{enumerate}[label=(\alph*)]
\item Then there exists some positive constant $C_{\mu}$ such that $\norm{S_{\nu_n}-S_{\mu}} \leq C_{\mu}W_2(\mu,\nu_n).$  In particular, convergence of a sequence of measures in the Wasserstein space implies the convergence of their frame operators. That is, the map $\mu\mapsto S_{\mu}$ is continuous.
\item Moreover, if $\mu$ is a probabilistic frame, there exists $N$ sufficiently large such that $\forall n \geq N,$ $\nu_n$ is also a probabilistic frame.
\item Similarly, for $\mu\neq\delta_{\{0\}}\in P_2(\Rd)$, given the scaled frame operator $\overline{S}_{\mu}:=S_{\mu}/M_2(\mu),$ there exists some positive constant $C'_{\mu}$ such that $\norm{\overline{S}_{\nu_n}-\overline{S}_{\mu}} \leq C'_{\mu} W_2(\mu,\nu_n).$  In particular, convergence of a sequence of measures to a probabilistic frame in the Wasserstein space implies the convergence of their scaled frame operators.
\end{enumerate}
\end{proposition}
\begin{proof}
Let $\{\nu_n\}$ and $\mu$ be as above, and let $S_{\nu_n},$ $S_{\mu}$ denote the matrix representations of their respective frame operators which exist since the measures in question are in $P_2(\Rd).$ Since $\nu_n\longrightarrow\mu$ in $P_2(\Rd),$ for $n$ sufficiently large, $M_2(\nu_n)\leq 2 M_2(\mu).$ 
Then, for any $\gamma_n\in\Gamma(\nu_n,\mu),$
\begin{align*}
\norm{S_{\nu_n}-S_{\mu}} &= \max_{v\in S^{d-1}} v^{\top}(S_{\nu_n}-S_{\mu})v=  \max_{v\in S^{d-1}} \iint_{\RdxRd} (\ip{x}{v}^2-\ip{y}{v}^2)d\gamma_n(x,y)\\						 
						&\leq \left(\iint_{\RdxRd} \norm{x-y}^2 d\gamma_n(x,y)\right)^{\half} \cdot \left(\iint_{\RdxRd} \norm{x+y}^2 d\gamma_n(x,y)\right)^{\half}\\		
						&\leq  C_{2,\gamma_n}(\mu,\nu_n)  \cdot \left(\iint_{\RdxRd} (\norm{x}+\norm{y})^2 d\gamma_n(x,y)\right)^{\half}\\								&\leq  C_{2,\gamma_n}(\mu,\nu_n)  \cdot \left(\left(\iint_{\RdxRd} \norm{x}^2 d\gamma_n(x,y)\right)^\half+\left(\iint_{\RdxRd}\norm{y}^2 d\gamma_n(x,y)\right)^{\half}\right)\\		
						 &\leq  C_{2,\gamma_n}(\mu,\nu_n) \cdot (1+\sqrt{2})\sqrt{M_2(\mu)}
\end{align*}
where the second-to-last inequality is by Minkowski, and the last inequality holds for $n$ sufficiently large.
In particular, if we choose $\gamma_n\in\Gamma_0(\nu_n,\mu),$ where $\Gamma_0$ is the set of all optimal couplings between $\nu_n$ and $\mu,$ then for $n$ sufficiently large, 
\begin{equation}
\norm{S_{\nu_n}-S_{\mu}}\leq (1+\sqrt{2}) W_2(\mu,\nu_n) \cdot \sqrt{M_2(\mu)}
\end{equation}

This control on the spectrum of the frame operator allows us to prove part (b).

Now, assuming $\mu$ is a probabilistic frame,  the matrix representation of the frame operator $S_{\mu}$ is positive definite. Let the eigenvalues of $S_{\nu_n}$ be given by $\lambda_1(S_{\nu_n})\leq\dots\leq\lambda_d(S_{\nu_n}).$  Then
\begin{align*}
\lambda_1(S_\mu)&= \min_{v\in S^{d-1}}\ip{v}{S_{\mu}v}	= \min_{v\in S^{d-1}}\left(\ip{v}{S_{\mu}v}-\ip{v}{S_{\nu_n}v}+\ip{v}{S_{\nu_n}v}\right)&\\
				&\leq \max_{v\in S^{d-1}} \left(\ip{v}{S_{\mu}v}-\ip{v}{S_{\nu_n}v}\right) + \ip{x}{S_{\nu_n}x} \qquad & \forall x\in S^{d-1}\\
				&=\lambda_d(S_{\mu}-S_{\nu_n}) + \ip{x}{S_{\nu_n}x} \qquad &\forall x\in S^{d-1} 
\end{align*} 
Since the last statement above holds for all $x$ in $S^{d-1},$ it holds in particular for\\ $x_*:=\arg\min_{x\in S^{d-1}}\ip{x}{S_{\nu_n x}}.$  Hence $$\lambda_1(S_{\mu}) \leq \lambda_d(S_{\mu}-S_{\nu_n}) + \lambda_1(S_{\nu_n}).$$
Therefore,  since by definition $$|\lambda_d(S_{\mu}-S_{\nu_n})|\leq\norm{S_{\mu}-S_{\nu_n}}\rightarrow 0$$ as $\nu_n\rightarrow\mu$ in $P_2(\Rd),$ given $\alpha \in (0,1),$ we can choose $N$ such that $\forall n\geq N,$ $$|\lambda_d(S_{\mu}-S_{\nu_n})|<\alpha\cdot\lambda_1(S_{\mu}),$$ and for such $n,$ $$\lambda_1(S_{\nu_n})>(1-\alpha)\lambda_1(S_{\mu})>0.$$

Finally, with repeated applications of Minkowski, we can prove (c) for $\nu_n$ sufficiently close to $\mu\neq \delta_{\{0\}}$:
\begin{align*}
\norm{\overline{S}_{\nu_n}-\overline{S}_{\mu}} &\leq \left(\iint_{\RdxRd} \norm{\frac{x}{\sqrt{M_2(\nu_n)}}-\frac{y}{\sqrt{M_2(\mu)}}}^2 d\gamma_n(x,y)\right)^{\half} \cdot \left(\iint_{\RdxRd} \norm{\frac{x}{\sqrt{M_2(\nu_n)}}+\frac{y}{\sqrt{M_2(\mu)}}}^2 d\gamma_n(x,y)\right)^{\half}\\			
					&\leq 	 2\cdot\left(\iint_{\RdxRd}\left(\frac{\norm{x-y}}{\sqrt{M_2(\mu)}}+\frac{|\sqrt{M_2(\mu)}-\sqrt{M_2(\nu_n)|}}{\sqrt{M_2(\mu)}\cdot \sqrt{M_2(\nu_n)}}	\norm{x}\right)^2d\gamma_n(x,y)\right)^\half\\
					&\leq 	2\cdot \left[\left(\iint_{\RdxRd}\frac{\norm{x-y}^2}{M_2(\mu)}d\gamma_n(x,y)\right)^\half +\frac{|\sqrt{M_2(\mu)}-\sqrt{M_2(\nu_n)}|}{\sqrt{M_2(\mu)}}\right] \\	
						 &\leq \frac{4}{\sqrt{M_2(\mu)}} C_{2,\gamma_n}(\mu,\nu_n).
\end{align*}
\end{proof}

It should be noted that if $\mu = \delta_{\{0\}}$ then convergence of the frame operators holds, but convergence of the scaled frame operators does not. Take for example the two sequences given by $\nu^1_n = \delta_{\epsilon_i},$ where $\epsilon_i = [0,\dots,\epsilon,\dots,0],$ and $\nu^2_n$ given by the uniform distribution on the ball of radius $\epsilon$ in $\Rd$. Both measures converge in $P_2(\Rd)$ to $\mu,$ but $\norm{\overline{S}_{\nu^1_n}}=1$, while $\norm{\overline{S}_{\nu^2_n}}=\frac{1}{d}.$ However, it is clear for any $\mu\in P_2(\Rd),$ $\mu\neq\delta_{\{0\}},$ that $\norm{\overline{S}_{\mu}}\leq 1,$ and this bound will be sufficient for our needs. 

As one might expect, given a probabilistic frame $\mu,$ this control also allows us to obtain a lower limit on the distance in $P_2(\Rd)$ to the nearest tight frame.  
\begin{proposition}\label{nearest_tight_frame_bound}
Suppose $\mu$ is a probabilistic frame for $\Rd$ which is not tight.  Let $$\delta :=(\lambda_d(\mu)-\lambda_1(\mu))\in(0,\lambda_d(\mu)).$$  Then for any tight frame $\nu,$ $W_2(\mu,\nu)\geq \frac{\delta}{2(\sqrt{M_2(\mu)}+\sqrt{M_2(\nu)}}).$
\end{proposition}  
\begin{proof}
From Proposition \ref{frame_op_dist}, we know that measures close to $\mu$ in $P_2(\Rd)$ will have frame operators whose spectra are close to that of the frame operator of $\mu.$  Let $\nu$ be a tight frame with frame constant $A:=\frac{M_2(\nu)}{d}.$ Then 
\begin{equation}\label{delta_bound}
\max_k|\lambda_k(\mu)-\lambda_k(\nu)|=\max\{|\lambda_1(\mu)-A|,|\lambda_d(\mu)-A|\}\geq\frac{\delta}{2}.
\end{equation}
\noindent Moreover, for any $k\in \{1,...,d\},$ $|\lambda_k(\mu)-\lambda_k(\nu)|\leq\norm{S_{\nu}-S_{\mu}}.$ Therefore, since from the proof of Proposition \ref{frame_op_dist} we know that for any $\gamma \in \Gamma_0(\mu,\nu),$
\begin{align*}
\norm{S_{\nu}-S_{\mu}}	&\leq W_2(\mu,\nu)\cdot\left(\iint_{\RdxRd}(\norm{x}+\norm{y})^2d\gamma(x,y)\right)^{\half}\\
						&\leq W_2(\mu,\nu)\cdot\left(\sqrt{M_2(\mu)} + \sqrt{M_2(\nu)}\right),
\end{align*}
\noindent it follows from (\ref{delta_bound}) that $W_2(\mu,\nu)\geq \frac{\delta}{2(\sqrt{M_2(\mu)}+\sqrt{M_2(\nu)})}.$
\end{proof}
%
%
%
\noindent The identification of tight frames with minimizers of the frame potential holds also in the case of probabilistic frames.  To that end, we restate a version of \cite[Theorem 4.2]{MartinRTF} in the first part of the following theorem.

\begin{theorem}\label{tightnessEquality}
Let $\mu$ be a measure in $P_2(\Rd).$  The following bound holds for the probabilistic frame potential: $PFP(\mu)\geq \frac{M_2^2(\mu)}{d}.$
Further, if $\mu\neq\delta_{\{0\}},$ then $PFP(\mu)=\frac{M_2^2(\mu)}{d}$ if and only if $\mu$ is a tight probabilistic frame.
\end{theorem}
\begin{proof}

Clearly, minimizers exist. In particular, if $\mu$ is a tight probabilistic frame, then equality holds in the above claim, since the frame bound of a probabilistic tight frame $\mu$ is precisely $\frac{M_2(\mu)}{d}$ (\cite{MartinRTF}), and
$$PFP(\mu)=\iint_{\RdxRd}\ip{x}{y}^2d\mu(x)d\mu(y)=\int_{\Rd}\ip{S_{\mu}y}{y}d\mu(y)=\int_{\Rd}\frac{M_2(\mu)}{d}\norm{y}^2d\mu(y)=\frac{M_2^2(\mu)}{d}$$

Now, let the eigenvalues of $S_{\mu}$ be $\lambda_1\leq\dots\leq\lambda_d$ with $\lambda_1\leq\frac{M_2(\mu)}{d}\leq\lambda_d$ with a corresponding orthonormal basis of eigenvectors $\{v_i\}_{i=1}^d$ for $\Rd.$  Then
\begin{align*}
PFP(\mu)	&= \iint\ip{x}{y}^2d\mu(x)d\mu(y)\quad = \int\ip{y}{S_{\mu}y}d\mu(y) = \int\ip{y}{\sum\limits_{i=1}^d\lambda_iv_iv_i^{\top}y}d\mu(y)\\
			&=\sum\limits_{i=1}^d\lambda_i\int\ip{v_i}{y}^2d\mu(y)=\sum\limits_{i=1}^d\lambda_i\ip{v_i}{S_{\mu}v_i}\quad =\sum\limits_{i=1}^d\lambda_i^2
\end{align*}
But, by H\"{o}lder, 
\begin{align*}
\sum\limits_{i=1}^d\lambda_i^2	&\geq \frac{1}{d}\left(\sum\limits_{i=1}^d\lambda_i\right)^2\quad =\frac{M_2^2(\mu)}{d}
\end{align*}
with equality if and only if $\lambda_1 = \dots = \lambda_d,$ that is, if and only if $\mu$ is tight.				
\end{proof}

\subsection{The tightness potential}\label{sec: the tightness potential}
\noindent With Theorem \ref{tightnessEquality} established, the tightness potential is now ready to be formally defined. 
For $\mu\in P_2(\Rd)$, the tightness potential $TP(\mu)$ is given by 
\begin{align} 
TP(\mu)&= \overline{PFP}(\mu)-\frac{M_2(\mu)}{d}\label{TP_def1}
\end{align}

For technical reasons, we fix $r > 2$ and restrict the domain of this functional to measures with $r^{th}$ moment uniformly bounded by some constant $L$, with the value of the functional being defined as $+\infty$ elsewhere. We write this restriction on the domain as: $D(TP):= \{\mu\in P(\Rd)\textrm{ s.t. } M_{r}(\mu)\leq L\}$.

\noindent A related object is the tightness operator $T_{\mu}:\Rd\rightarrow\Rd$ given by 
\begin{align*}
T_{\mu}(x):=\left\{ \begin{array}{cc} 
                \overline{S}_{\mu}x-\frac{\overline{PFP}(\mu)}{2M_2(\mu)}x-\frac{1}{2d}x & \hspace{5mm} \mu\neq\delta_{\{0\}} \\ 
                0 & \hspace{5mm} \text{else} \\
                \end{array} \right.
\end{align*}
\noindent where we note that $\overline{S}_{\mu}x-\frac{\overline{PFP}(\mu)}{2M_2(\mu)}x-\frac{1}{2d}x = \int_{\Rd}\left[\frac{\ip{x}{y}}{M_2(\mu)}y-\frac{\overline{PFP}(\mu)}{2M_2(\mu)}x-\frac{1}{2d}x\right]d\mu(y)$ when $\mu$ is not a delta mass at the origin.
\noindent We immediately obtain:
\begin{proposition}\label{tp_pot_bound_tp_op}
For a measure $\mu\in D(TP)$, $\norm{T_{\mu}}\leq (\frac{d-1}{2}TP(\mu))^{\half}$. Moreover, the tightness potential is zero if and only if $\mu$ is tight or $\mu=\delta_{\{0\}}$.
\end{proposition}
\begin{proof}
Given $\mu\in D(TP)$, $\mu\neq\delta_{\{0\}}$, let $0\leq\lambda_1\leq\lambda_2\leq\dots\leq\lambda_d$  be the eigenvalues of $S_\mu$. Noting that $M_2(\mu)=\sum\limits_{i=1}^d\lambda_i$, we have the following equivalence for the tightness potential:
\begin{align*}
TP(\mu)	&= \iint \frac{\ip{x}{y}^2}{M_2(\mu)} - \frac{1}{d}\norm{x}^2d\mu(x)d\mu(y)\\
		&= Tr(\overline{S}_\mu S_{\mu}) - \frac{1}{d}M_2(\mu)\\
		&= \left[\sum\limits_{i=1}^d\lambda_i^2-\frac{1}{d}\left(\sum\limits_{i=1}^d\lambda_i\right)^2\right]/\sum\limits_{i=1}^d\lambda_i\\
		&=\frac{1}{d}\sum\limits_{i=1}^d\sum\limits_{j>i}(\lambda_i-\lambda_j)^2/\sum\limits_{i=1}^d\lambda_i
\end{align*}
Rewriting $\tlambda_i=\frac{\lambda_i}{\sum\limits_{i=1}^d\lambda_i}$, we have $\norm{T_{\mu}}=\max\{\tlambda_d-\half\sum\limits_{i=1}^d\tlambda_i^2-\frac{1}{2d},\frac{1}{2d}+\half\sum\limits_{i=1}^d\tlambda_i^2-\tlambda_1\}$.  Without loss of generality, let $\norm{T_{\mu}}=\tlambda_d-\half\sum\limits_{i=1}^d\tlambda_i^2-\frac{1}{2d}.$ Then 

$$\norm{T_{\mu}} \leq \tlambda_d-\frac{1}{2d}-\frac{1}{2d} =\tlambda_d-\frac{1}{d}.$$

Then, by Cauchy's inequality, noting that $\lambda_k-\lambda_j\geq 0$ if $k>j$,
\begin{align*}
\sqrt{\frac{d-1}{2}}\left(\frac{1}{d}\sum\limits_{i=1}^d\sum\limits_{j>i}(\lambda_j-\lambda_i)^2/\left(\sum\limits_{i=1}^d\lambda_i\right)^{2}\right)^{\half}	 &=\sqrt{\frac{d-1}{2}}\left[\sum\limits_{i=1}^d\sum\limits_{j>i}(\tlambda_j-\tlambda_i)^2\right]^{\half}\\
						&\geq \frac{1}{d}\sum\limits_{i=1}^d\sum\limits_{j>i}(\tlambda_j-\tlambda_i)\\
						&=\frac{1}{d}[(d-1)\tlambda_d-\sum\limits_{j<d}\tlambda_j+\sum\limits_{k=1}^{d-1}\sum\limits_{j<k}(\tlambda_k-\tlambda_j)\\
						&\geq \tlambda_d-\frac{1}{d}\sum\limits_{j=1}^d\tlambda_j = \tlambda_d-\frac{1}{d}
\end{align*}
From the above, we see that $(\frac{d-1}{2M_2(\mu)})\cdot TP(\mu)\geq\norm{T_{\mu}}^2,$ with equality if and only if $\lambda_i=\lambda_j$ $\forall i,j$.

Clearly, if $\mu$ is a tight probabilistic frame or $\mu=\delta_{\{0\}}$, then $TP(\mu)=0$. If $\mu$ is not tight and $\mu\neq\delta_{\{0\}}$, then $\norm{T_{\mu}}^2>0$, so that by the above, $TP(\mu)>0$.
\end{proof}

\subsection{Construction of gradient flows for the tightness potential}\label{sec: construction}
Most approaches to establishing the well-posedness of a gradient flow for a particular potential use the convexity or $\lambda$-convexity of the functional, if it can be established, as in \cite{CLM2010}, where the authors work with a class of nonlinear interaction potentials.  While the tightness potential is also an interaction potential, it is not straightforward to establish $\lambda$-convexity on $P_2(\Rd).$ (And, in fact, the potential defined by $\mu\rightarrow M_2(\mu)TP(\mu)$ can explicitly be shown not to be $\lambda$-convex.) Therefore, we proceed down the path outlined in the introduction, employing the minimizing movement scheme and related existence results.  For this approach, we establish a few facts about the frame and tightness potentials. The first key result is the differentiability of the tightness potential.

\begin{proposition}\label{tp_differentiable}
The tightness potential $TP(\mu) :=\oPFP(\mu) - \frac{1}{d}M_2(\mu)$ is a strongly differentiable function on $P_2(\Rd).$ Moreover, an element of its subdifferential is $(i,4T_{\mu})_{\hash}\mu,$ and for any $\mu,\nu\in D(TP),$ we have $$TP(\nu)-TP(\mu) = \iiint_{\RdxRdxRd}\ip{y}{z-x}d\beta(x,y,z) + o(C_{2,\beta}(\mu,\nu)),$$ for any $\beta\in P(\RdxRdxRd)$ satisfying $\pi^{1,2}_{\hash}\beta = (i,4T_{\mu})_{\hash}(\mu)$ and $\pi^3_{\hash}\beta = \nu$ (where $i_{\hash}\mu$ denotes the identity pushforward of $\mu$ and the precise form of the term $o(C_{2,\beta}(\mu,\nu))$ is bounded in absolute value by\\ $(15+4\sqrt{2})C^2_{2,\beta}(\mu,\nu)$).
\end{proposition}

\begin{proof} 

We first prove that powers of the second moment $M_2(\mu):=\int_{\Rd}\norm{x}^2d\mu(x)$ are strongly differentiable functions on $P_2(\Rd).$  

Take $\mu$ in $P_2(\Rd).$ Consider $\gamma=(i,2I)_{\hash}\mu\in P_2(\RdxRd).$ Given some $\nu$ in $P_2(\Rd),$ take any $\beta\in\Gamma(\gamma,\nu).$  Then
\begin{align*}
M_2(\nu)-M_2(\mu) &= \int\norm{z}^2d\nu(z)-\int_{\Rd}\norm{x}^2d\mu(x)\\
						&=\iiint\ip{y}{z-x}d\beta(x,y,z)+\iiint\norm{x-z}^2d\beta(x,y,z).
\end{align*}
\noindent Therefore, $M_2(\nu)-M_2(\mu) = \iiint_{\RdxRdxRd}\ip{y}{z-x}d\beta(x,y,z) + o(C_{2,\beta}(\mu,\nu))$ for $\nu$ sufficiently close to $\mu.$ Thus by Definition \ref{strong_p_subdifferential}, $\gamma$ is an element of the strong Fr\'{e}chet subdifferential of the second moment.

The extension to higher powers is easily proven by induction, showing that $\gamma=(i,kM_2^{k-1}(\mu)I)_{\hash})\mu$ is an element of the strong subdifferential of $M_2^k(\mu).$  

Next we prove strong differentiability in $P_2(\Rd)$ of  the scaled frame potential, defined as:

\begin{align*}
\overline{PFP}(\mu) = \left\{ \begin{array}{cc} 
                \frac{1}{M_2(\mu)}\iint_{\RdxRd}\ip{x}{y}^2d\mu(x)d\mu(y) & \hspace{5mm} \mu\neq\delta_{\{0\}} \\
                0 & \hspace{5mm} \text{else.} \\
                \end{array} \right.
\end{align*}

We write $\ooPFP(\mu)$ for $\frac{\overline{PFP}(\mu)}{M_2(\mu)}=\frac{PFP(\mu)}{M_2^2(\mu)}$ where $\mu\neq\delta_{\{0\}}$. If $\mu\neq\delta_{\{0\}},$ consider $\gamma=(i,4\overline{S}_{\mu}-2\ooPFP(\mu))_{\hash}\mu\in P_2(\RdxRd).$ Given some $\nu\in P_2(\Rd),$ take any $\beta\in\Gamma(\gamma,\nu).$ Then, noting that $\iiint_{\RdxRdxRd}\ip{y}{x}d\beta(x,y,z)=2\oPFP(\mu)$:

\begin{flalign*}
& \oPFP(\nu)-\oPFP(\mu) =\int_{\Rd}\ip{\oS_{\nu}z}{z} + \ip{y}{z} -\ip{y}{z}-\ip{y}{x}+\half\ip{y}{x}d\beta(x,y,z)\\
				& \textrm{\hspace{1.5cm}} =\iiint_{\RdxRdxRd}\ip{y}{z-x}d\beta(x,y,z)+\iiint_{\RdxRdxRd}\ip{\oS_{\nu}z}{z}-\ip{y}{z}+\frac{1}{2}\ip{y}{x}d\beta(x,y,z). 
\end{flalign*}

\noindent  Again, considering the second term in the preceding equation, for $\nu$ sufficently close to $\mu,$ say, such that $M_2(\nu)\leq 2 M_2(\mu),$ we have:
\begin{flalign*}
&\left|\: \iiint_{\RdxRdxRd}\ip{\oS_\nu z}{z}-\ip{y}{z}+\frac{1}{2}\ip{y}{x} d\beta(x,y,z) \:\right| \\
&= \left|\:\iiint_{\RdxRdxRd} \ip{\oS_\nu z}{z} -4\ip{\oS_\mu x}{z} + 3\ip{\oS_\mu x}{x} -\ip{\oS_\mu x}{x}+2\ip{\ooPFP(\mu)x}{z}\right.\\
&\left.\textrm{\hspace{1.5cm}}-\ip{\ooPFP(\mu)x}{x}d\beta(x,y,z)\:\right|\\
	&= \left| \iiint_{\RdxRdxRd} \underbrace{\ip{\oS_\nu(z-x)}{z-x} + \ip{\oS_\mu(z-x)}{z-x}+2\ip{(\oS_\mu-\oS_\nu)x}{x-z}}_A  \right.\\
	 &\textrm{\hspace{1.5cm}}\left.+\ip{\oS_\nu x}{x}-\ip{\oS_\mu z}{z} -\ip{\oS_\mu x}{x}+2\ip{\ooPFP(\mu)x}{z}-\ip{\ooPFP(\mu)x}{x}d\beta(x,y,z)\:\right|\\
	 & = \left|\iiint_{\RdxRdxRd}\underbrace{A + \ooPFP(\mu)\ip{x-z}{z-x}}_B + \frac{[M_2(\mu)-M_2(\nu)]}{M_2(\mu)M_2(\nu)}\ip{S_\nu x}{x}-\ip{\oS_\mu x}{x} + \ooPFP(\mu)\norm{z}^2 d\beta(x,y,z)\right|\\ 
	 & = \left| \iiint_{\RdxRdxRd} B + \frac{[M_2(\mu)-M_2(\nu)]}{M_2(\mu)M_2(\nu)}\ip{S_\nu x}{x} + \frac{1}{M_2^2(\mu)}\left[PFP(\mu)\norm{z}^2 - M_2(\mu)\ip{S_\mu x}{x}\right] d\beta(x,y,z)\right|\\
	 	 &  = \left|\iiint_{\RdxRdxRd} B + \frac{[M_2(\mu)-M_2(\nu)]}{M_2(\mu)M_2(\nu)}\ip{S_\nu x}{x} + \frac{1}{M_2^2(\mu)}\left[M_2(\nu) - M_2(\mu)\right]\ip{S_\mu x}{x}) d\beta(x,y,z)\right|\\
	 	 & = \left| \iiint_{\RdxRdxRd} B + \frac{[M_2(\mu)-M_2(\nu)]}{M_2(\mu)}\ip{(\oS_\nu - \oS_\mu) x}{x} d\beta(x,y,z)\right|\\
	 	 &\leq \iiint_{\RdxRdxRd} \norm{\oS_\nu}\norm{z-x}^2 + \norm{\oS_\mu}\norm{z-x}^2 +2\norm{\oS_\mu-\oS_\nu}\norm{x}\norm{z-x}\\ 
	 	 &\textrm{\hspace{1.5cm}} + \ooPFP(\mu)\norm{z-x}^2+\frac{|M_2(\mu)-M_2(\nu)|}{M_2(\mu)}\norm{\oS_\nu-\oS_\mu}\norm{x}^2d\beta(x,y,z)\\
	 & \leq (\norm{\overline{S}_{\nu}}+\norm{\overline{S}_{\mu}})C_{2,\beta}^2(\mu,\nu) + 2\norm{\overline{S}_{\mu}-\overline{S}_{\nu}}\cdot \sqrt{M_2(\mu)}\cdot C_{2,\beta}(\mu,\nu)  \\
	 &\textrm{\hspace{1.5cm}} + C^{2}_{2,\beta}(\mu,\nu) + (1+\sqrt{2})\sqrt{M_2(\mu)}C_{2,\beta}(\mu,\nu)\norm{\oS_\mu-\oS_\nu}\\
 &\leq (2+2\cdot 4+1+(1+\sqrt{2})\cdot 4) C^2_{2,\beta}(\mu,\nu) \\
  &\leq (15 + 4\sqrt{2}) C^2_{2,\beta}(\mu,\nu) \\
\end{flalign*} 

\noindent where we employ the bound on $\norm{\oS_\mu-\oS_\nu}$ derived in Proposition \ref{frame_op_dist}. Thus by Definition \ref{strong_p_subdifferential}, $\gamma$ is an element of the strong Fr\`{e}chet subdifferential of $\oPFP(\mu)$ on $P_2(\Rd)$.

Finally, if $\mu =\delta_{\{0\}},$ we can consider $\gamma=(i,\frac{2I}{d})_{\hash}\mu\in P_2(\RdxRd).$ Given some $\nu\in P_2(\Rd),$ take any $\beta\in\Gamma(\gamma,\nu).$ Then, as we have defined $\overline{PFP}(\delta_{\{0\}}) = 0$:

\begin{align*}
\overline{PFP}(\nu)-\overline{PFP}(\mu)	&=\int_{\Rd}\ip{\overline{S}_{\nu}z}{z}d\nu(z)-0\\
				&=\iiint_{\RdxRdxRd}\ip{y}{z-x}d\beta(x,y,z)+\iiint_{\RdxRdxRd}\ip{\overline{S}_{\nu}z}{z}d\beta(x,y,z)
\end{align*}

\noindent and again considering the second term in the preceding line,

\begin{align*}
\left| \iiint_{\RdxRdxRd}\ip{\overline{S}_{\nu}z}{z} d\beta(x,y,z)\right| &\leq C^2_{2,\beta}(\mu,\nu).
\end{align*} 

So again by Definition \ref{strong_p_subdifferential}, $\gamma$ is an element of the strong Fr\`{e}chet subdifferential of $\oPFP(\mu)$ on $P_2(\Rd)$.

Thus, given $\mu\in P_2(\Rd)$, take $\gamma=(i,4T_{\mu})_{\hash}\mu$.  Then by the differentiability results shown above, $\gamma$ clearly satisfies Equation \eqref{strong_frechet}, and the tightness potential is a strongly differentiable function on its domain. 
\end{proof}

\begin{proposition}\label{tp_regular}
The tightness potential is a regular functional on $D(TP).$
\end{proposition}

\begin{proof}
Let $TP$ denote the tightness potential. Suppose that $\eta_n\in\strongsub{TP}{\mu_n}$ is a sequence of extended strong subdifferentials of $TP$ for a sequence of measures $\mu_n\in D(TP)$ satisfying:
$$ TP(\mu_n)\rightarrow\varphi\in\R,\quad\mu_n\rightarrow\mu\quad\textrm{in }P_2(\Rd),$$ $$\sup_{n}|\eta_n|_{2,2}<\infty,\quad\eta_n\rightarrow\eta\quad\textrm{in }P(\Rd\times\Rd).$$
We must show that $\varphi = TP(\mu)$ and $\eta\in\extendedsub{TP}{\mu}.$

First, we show that $ TP(\mu_n)\rightarrow TP(\mu)$. If $\mu=\delta_{\{0\}},$ this is trivial, so suppose not. By the differentiability result in Proposition \ref{tp_differentiable}, for any $\gamma_n\in\Gamma(\mu,\mu_n)$, and in particular for $\gamma_n\in\Gamma_0(\mu,\mu_n)$,
\begin{align*}
| TP(\mu_n)- TP(\mu)|	&=\left|\:\iint_{\RdxRd}\ip{4T_\mu x}{y-x}d\gamma_n(x,y) + o(W_2(\mu_n,\mu))\:\right|\\
						&\leq 4\norm{T_{\mu}}M_2(\mu)W_2(\mu,\mu_n)+o(W_2(\mu_n,\mu)).
\end{align*}
Thus, as $\mu_n\rightarrow\mu$ in $P_2(\Rd)$, $W_2(\mu,\mu_n)\rightarrow 0$, and $ TP(\mu_n)\rightarrow TP(\mu)$.  Hence, $\varphi= TP(\mu)$. 

Second, we consider the sequence of strong subdifferentials, $\eta_n$.  Noting that $\mu_n$ converges to $\mu$ in $P_2(\Rd),$ $\sup_n|\eta_n|_{1,2} = \sup_n M_2(\mu_n) < \infty.$ Then since, by assumption, $\sup_n|\eta_n|_{2,2}<\infty,$ the sequence is tight (in the probability sense, i.e., given any $\epsilon>0,$ $\exists K_{\epsilon}$ compact in $\RdxRd$ s.t. $\forall n,$ $\eta_n(K_{\epsilon}^C) < \epsilon$). 

Given any $\mu^0\in P_2(\Rd)$, we pick a sequence $\{\nu_n\}\in\Gamma_0(\eta_n,\mu^0)$.  Then we have for all $n\in\N$,
\begin{equation}\label{strong_frechet2}
 TP(\mu^0)- TP(\mu_n)\geq \iint\ip{x_2}{x_3-x_1}d\nu_n(x_1,x_2,x_3) + o(C_{2,\eta_n}(\mu_n,\mu^0)).
\end{equation}
Let $\nu\in\Gamma_0(\eta,\mu^0)$ be a limit point of $\nu_n$ in $P(\Rd\times\Rd\times\Rd)$.  (Its existence follows from tightness of the marginals by \cite[Lemmas 5.2.2 and 5.1.12]{AGS2005}.) Then as $n\rightarrow\infty$, the left-hand side of Equation \eqref{strong_frechet2} converges to $ TP(\mu^0)- TP(\mu)$ by our first result.

As for the right-hand side, we write,
\begin{equation*}
\iiint\ip{x_2}{x_3-x_1}d\nu_n=\iiint\ip{x_2}{x_3}d\pi^{2,3}_{\hash}\nu_n - \iiint\ip{x_2}{x_1}d\pi^{1,2}_{\hash}\nu_n,
\end{equation*}
noting that the same decomposition can be done for the integral with respect to $\nu$, the limit point.

Then applying \cite[Lemma 5.2.4]{AGS2005} to $\pi^{2,3}_{\hash}\nu_n$, whose second marginal, $\mu_0\in P_2(\Rd)$ clearly has a [uniformly] integrable second moment, and to $\pi^{1,2}_{\hash}\nu_n$, whose second marginals, $\mu_n$ are converging in $P_2(\Rd)$ and hence have U.I. second moments, we conclude that
\begin{align*}
\lim_{n\rightarrow\infty}\iiint\ip{x_2}{x_3-x_1}d\nu_n	&= \lim_{n\rightarrow\infty}\iint\ip{x_2}{x_3}d\pi^{2,3}_{\hash}\nu_n - \lim_{n\rightarrow\infty}\iiint\ip{x_2}{x_1}d\pi^{1,2}_{\hash}\nu_n\\
														& =	\iint\ip{x_2}{x_3}d\pi^{2,3}_{\hash}\nu - \iiint\ip{x_2}{x_1}d\pi^{1,2}_{\hash}\nu\\
														&= \iiint\ip{x_2}{x_1-x_3}d\nu.
\end{align*}
Finally, by the narrow lower semicontinuity of $W_2$ in Hilbert spaces (c.f. \cite[Lemma 7.1.4]{AGS2005}), $$W_2(\mu,\mu^0)\leq\liminf_{n\rightarrow\infty}W_2(\mu^0,\mu_n)$$ and the fact that $$C_{2,\eta_n}(\mu_n,\mu^0)\geq W_2(\mu_n,\mu^0),$$ we conclude that 
\begin{equation*}
 TP(\mu^0)- TP(\mu)\geq \iint\ip{x_2}{x_3-x_1}d\nu(x_1,x_2,x_3) + o(W_2(\mu,\mu^0)),
\end{equation*}
so that $\eta\in\extendedsub{ TP}{\mu}$.

\end{proof}

Moreover, $\gamma :=(i,4T_{\mu})_{\hash}\mu$ is the minimal selection in the strong subdifferential for probability measures of interest: 
\begin{proposition}\label{tp_minimal_selection}
Given $\mu\in D(TP),$ where $\mu$ is not tight and $\mu\neq\delta_{0},$ $\gamma:=(i,4T_{\mu})_{\hash}\mu$ attains the minimum metric slope, i.e., $|\gamma|_{2,2}=min\{|\xi|_{2,2}:\xi\in\partial TP(\mu)\}.$ 
\end{proposition}
\begin{proof}
Recalling Definition \ref{metric_slope}, Lemma \ref{minimal_selection}, and the definition given in Equation\eqref{partial_moment}, since $TP$ is regular, it is sufficient to show that $|\gamma|_{2,2}=|\partial TP|(\mu).$  It is clear by definition of subdifferentiability that $|\gamma|_{2,2}\geq|\partial TP|(\mu).$  

Letting $g_t(x)=x+ 4tT_{\mu}x,$ and $\alpha_t\in\Gamma(\mu,{g_t}_{\hash}\mu),$
\begin{align*}
|\partial TP|(\mu) 	&=	\limsup_{W_2(\mu,\nu)\rightarrow 0}\frac{(TP(\mu)-TP(\nu))^+}{W_2(\mu,\nu)}\\
					&\geq \lim_{t\rightarrow 0}\frac{(TP(\mu)-TP({g_t}_{\hash}\mu))^+}{W_2(\mu,{g_t}_{\hash}\mu)}\\
					& \geq \lim_{t\rightarrow 0}\frac{TP(\mu)-TP({g_t}_{\hash}\mu)}{C_{2,\alpha_t}(\mu,{g_t}_{\hash}\mu)}\\
					& = \lim_{t\rightarrow 0}\frac{\iint_{\RdxRd}\ip{4 T_{\mu}x}{y-x}d\alpha_t(x,y)+ o(C_{2,\alpha_t}(\mu,{g_t}_{\hash}\mu))}{C_{2,\alpha_t}(\mu,{g_t}_{\hash}\mu)}\\
					& = \left(\int_{\Rd}\norm{4 T_{\mu}x}^2d\mu(x)\right)^\half = |\gamma|_{2,2}
\end{align*}
\noindent since $C_{2,\alpha}(\mu,{g_t}_{\hash}\mu)=t\left(\int_{\Rd}\norm{4 T_{\mu}x}^2d\mu(x)\right)^\half,$ and $\lim_{t\rightarrow 0}\frac{o(C_{2,\alpha_t}(\mu,(g_t)_{\hash}\mu))}{C_{2,\alpha_t}(\mu,{g_t}_{\hash}\mu)}=0.$
\end{proof}


\subsection{Well-posedness of the minimization problem}\label{sec: well-posedness}

Because the standard machinery of $\lambda$-convexity does not apply for our potential, we  establish the well-posedness of the problem of constructing gradient flows for the tightness potential following the approach of \cite[Chapter 11.3]{AGS2005}, using in particular Lemma \ref{existence_flows_ags}.  This machinery does not provide a proof of uniqueness, which \textit{a priori} seems natural, since, given a nontight probabilistic frame, there are a multitude of tight probabilistic frames outside a ball of the radius established in Proposition \ref{nearest_tight_frame_bound}.

First, we state our main result:
\begin{theorem}\label{existence_flows_tp}
Gradient flows exist for the tightness potential on $P_2(\Rd),$ i.e. for every initial datum $\mu_0\in P_2(\Rd)$  in the domain of $TP$, each sequence of discrete solutions $\overline{M}_{\tau_k}$ of the variational scheme admits a subsequence such that:
\begin{enumerate}
\item $\overline{M}_{\tau_k}(t)$ narrowly converges in $P(\Rd)$ to $\mu_t$ locally uniformly in $[0,\infty)$, with \\$\mu_t\in AC_{loc}^2([0,\infty);P_2(\Rd))$.
\item $\mu_t$ is a solution of the gradient flow equation $$v_t =-\partial^0 TP(\mu_t),\quad\norm{v_t}_{L^2(\mu_t;\Rd)}=|\mu^{'}|(t), \textrm{ for a.e. } t>0$$
with $\mu_t\rightarrow\mu_0$ as $t\downarrow 0,$ where $v_t(x)=-4T_{\mu_t}(x)$ is the tangent vector to the curve $\mu_t.$
\item The energy inequality $$\int_a^b\int_{\Rd}|v_t(x)|^2d\mu_t(x)dt+TP(\mu_b)\leq TP(\mu_a)$$ holds for every $b\in [0,\infty)$ an $a\in [0,b)\setminus\mathcal{N}$, where $\mathcal{N}$ is a $\mathcal{L}^1$-negligible subset of $(0,\infty).$
\end{enumerate}
\end{theorem}
\begin{proof}
This will follow from Proposition \ref{tp_regular}, Proposition \ref{tp_lsc}, and Proposition \ref{tp_compact} by Lemma \ref{existence_flows_ags}, with Proposition \ref{tp_minimal_selection} providing the identification of the minimal selection (as defined in Lemma \ref{minimal_selection}) with the measure $(i,4T_{\mu})_{\hash}\mu$. 

\end{proof}
\begin{remark}
As a result of this theorem, we have an algorithm derived from the Minimizing Movement Scheme that leads to approximations to tight frames. In practice, without modification to the potential, these are the closest tight frames to the initial point. This algorithm does not require any restriction on the norms of the supporting vectors.
\end{remark}

\begin{proposition}\label{tp_lsc}
The tightness potential is lower semicontinuous with respect to narrow convergence in $P_2(\Rd).$
\end{proposition}

\begin{proof}
Take a narrowly converging sequence $\mu_n\rightarrow\mu\neq\delta_0$ in $P_2(\Rd).$ On the domain of the functional, the $r^{th}$ moments of any such sequence will be uniformly bounded, so that in fact the sequence will converge in $P_2(\Rd).$ Then since the function $f(x,y) :=|\ip{x}{y}|^2-\frac{1}{d}\norm{x}^2\norm{y}^2$ is bounded in absolute value by $\frac{d-1}{d}\norm{x}^2\norm{y}^2,$ it will be uniformly integrable by \cite[Lemma 5.1.7]{AGS2005} with respect to $\mu_n\times\mu_n$ and 
\begin{align*}
\lim_{n\rightarrow\infty} TP(\mu_n)\cdot M_2(\mu_n) &= \lim_{n\rightarrow\infty}\iint f(x,y)d\mu_n(x)d\mu_n(y) \\
&= \iint f(x,y)d\mu(x)d\mu(y) \\
&= TP(\mu)\cdot M_2(\mu).
\end{align*}
Since the second moments will similarly converge, we see that $\lim_{n\rightarrow\infty} TP(\mu_n) = TP(\mu).$ The result also holds if $\mu=\delta_0,$ since $0\leq TP(\mu_n) \leq\frac{d-1}{d}M_2(\mu_n)$ guarantees $\lim_{n\rightarrow\infty} TP(\mu_n) = 0.$
\end{proof}

Finally, recalling the definition of the sublevel sets of a functional $\phi:P_2(\Rd)\rightarrow \mathbb{R}:$ $$\Sigma_m(\phi) :=\left\{\mu\in P_2(\Rd):\quad \phi(\mu) \leq m,\quad M_2(\mu)\leq m\right\},$$

\noindent we prove the last pillar necessary to establish the above result.
\begin{proposition}\label{tp_compact}
The sublevels of the tightness potential are compact in $P_2(\Rd)$ with respect to the narrow convergence.
\end{proposition}
\begin{proof}
This holds trivially by the compactness of the intersection of the domain of the functional and the closed $M_2$-ball  because of the bound (for $\mu$ in the domain of the functional)
\begin{equation*}
TP(\mu) = \frac{PFP(\mu)}{M_2(\mu)}-\frac{M_2(\mu)}{d}\leq \frac{d-1}{d}M_2(\mu)\leq \frac{m(d-1)}{d}.
\end{equation*}
\end{proof}

As a result of Theorem \ref{existence_flows_tp}, given a probabilistic frame $\mu_0\in P_2(\Rd)$, there exists a flow $\phi_t$ such that $\phi_0(x)=x$ and $$\partial_t\phi_t(x)=v_t(\phi_t(x))=-4T_{\mu_t}\phi_t(x),$$ and $\mu_t = (\phi_t)_{\hash}\mu_0$ is a solution to the continuity equation with $$\int_a^b\int_{\Rd}|v_t(x)|^2d\mu_t(x)dt+TP(\mu_b)\leq TP(\mu_a)$$ for every $b\in [0,\infty)$ an $a\in [0,b)\setminus\mathcal{N}$, where $\mathcal{N}$ is a $\mathcal{L}^1$-negligible subset of $(0,\infty).$ Therefore, as long as the integrand of the first term in the preceding inequality is a.e. nonzero with respect to $\mu_t,$ then for any $t\in [a,b],$ the tightness potential is strictly decreasing on that interval.  Since $T_{\mu_t}$ is nonzero unless $\mu_t$ is tight or zero, the tightness potential will decrease until $\mu_t$ is tight unless $\phi_t\equiv 0$ on the support of $\mu_t$ for some $t\in [a,b],$ i.e., unless there is some ``singularity'' in the path.

\subsection{Extension to Higher-Order Potentials and Related Work}\label{sec: the fourth and higher potentials}

Probabilistic $p$-frames are those probability measures on $\Rd$ for which there exist $0<A\leq B<\infty$ such that for all $y\in\Rd$, $$A\norm{y}^p\leq\int_{\Rd}|\ip{x}{y}|^pd\mu(x)\leq B\norm{y}^p.$$ For $p\in(0,\infty)$ and $\mu\in P_p(\Rd),$ we can define the $p$-frame potential, $PFP_p(\mu)$ by $$PFP_p(\mu)=\iint_{\RdxRd}|\ip{x}{y}|^pd\mu(x)d\mu(y).$$

\noindent We highlight two results proved in \cite{MartinKassoPframe}. On the one hand, when restricted to probability measures on the unit sphere and for $p\in (0, 2)$ the minimizers of $PFP_p$ are discrete measures supported by orthonormal bases \cite[Theorem 4.9]{MartinKassoPframe}. On the other hand, when $p>2$ is even, \cite[Theorem 4.10]{MartinKassoPframe}  proved that the minimizers of this potential among probabilistic frames supported on $S^{d-1}$ are precisely the probabilistic tight $p$-frames: those probabilistic $p$-frames for which $A=B$ in the inequality above.

And, as with the case $p=2$, we have a lower bound, which generalizes the lower bound given in \cite{MartinKassoPframe}:
\begin{theorem}\label{p_frame_bound}
Take $\mu\in P_p(\Rd)$ for $p\geq 2$ an even number. Then $$PFP_p(\mu)\geq\frac{(p-1)(p-3)\cdots1}{(d+p-2)(d+p-4)\cdots d}\left(\int_{\Rd}\norm{x}^pd\mu(x)\right)^2$$ with equality if and only if $\mu$ is a tight $p$-frame.
\end{theorem}

\noindent We note that this result was also proven previously by Sidel'nikov, as cited in \cite[Theorem 1]{KotelinaPevnyi}. In fact, for even $p=2k,$ the finitely-supported minimizers of the potential include those for which, with their support denoted $X,$ $X\bigcup -X$ is a spherical $2k$-design (\cite{XuXu,GoethalsSeidel,Sidelnikov1974}). 

In the discrete setting and when $p>2,$ the above bound may not be tight, especially when the cardinality of the support of the minimizer and the dimension of the space are incompatible with the existence of an ETF, as proved in \cite{MartinKassoPframe}. The minimum is attained, for example, when the cardinality of the support is $N=d+1,$ and in this case the minimum above is achieved by a measure supported on the vertices of the regular simplex. However, for a given dimension, there are restrictions on the possible cardinalities for equiangular line sets and ETFs which are still being worked out (see, e.g., \cite{CasazzaRedmondTremain}). Moreover, as noted in \cite{XumeiKassoEtAl}, it is nontrivial, even for small dimension, to find spherical $k$-designs for large $k.$ For these reasons, finding minimizers of the $p$-frame potential for higher values of $p$ and understanding their geometry is of great interest \textemdash hence our desire to develop an gradient-flow-based, numerically-implementable approach. In addition, we note that Theorem~\ref{p_frame_bound} also addressed the case where the probability measures are not supported on the unit sphere.

We assert that Wasserstein gradient flows in $P_p(\Rd)$ for a normalized version of this potential could be constructed to obtain tight $p$-frames for $p$ even (as we did above for $p=2$). Indeed, we have computed a gradient for such a potential and simulations show convergence to, e.g., equiangular tight frames for $p=4,6,8.$ As noted above, the complexity in this regime lies in the fact that there is a complex interplay between dimension and cardinality for any finitely-supported measure. Further, the lower bound for the $p$-probabilistic frame potential for discrete frames (as, indeed, any finite approximation must be), is not the same as that given in Theorem \ref{p_frame_bound}. Therefore, the form of the objective function must not use any assumption about the minimum value. We look forward to sharing these results in a subsequent publication.

Finally, we end with two points on the form of the gradient for quantities related to such a potential. First, we can state the gradient for the $pth$ moment (even $p$) in $P_p(\Rd)$.
\begin{claim}
The $p$-th moment is a differentiable function on $P_p(\Rd)$ for even $p = 2k > 2$ with an element of the subdifferential being $\gamma=(i,h_p^{\mu})_{\hash}\mu\in P_{p,q}(\RdxRd), q=\frac{p-1}{p},$ where $h_p^{\mu}(x):=2\norm{x}^{p-2}x$.
\end{claim}


And using the extension of the Otto calculus from $P_2(\Rd)$ to $P_p(\Rd)$, we can show the following, the proof of which follows from Definition \ref{strong_p_subdifferential}:

\begin{proposition}
The $p$-frame potential is a differentiable function on $P_p(\Rd)$ for even $p > 2$ with an element of the subdifferential being $\gamma=(i,g_p^{\mu})_{\hash}\mu\in P_{p,q}(\RdxRd), q=\frac{p-1}{p},$ where $$g_p^{\mu}(x):=2p\int\ip{x}{u}^{p-1}ud\mu(u).$$
\end{proposition}

These two subdifferentials are sufficient to design a gradient descent scheme for frame potentials of even order.


\section{Acknowledgments}

For this research, C.~Wickman did not receive any specific grant from funding agencies in the public, commercial, or not-for-profit sectors. K.~A.~Okoudjou  was partially supported by a grant from the Simons Foundation $\# 319197$, by ARO grant W911NF1610008, and the National Science Foundation grant DMS 1814253.

\bibliographystyle{acm}
\bibliography{Bibliography} 

\end{document}